\documentclass[reqno,11pt]{amsart}
\usepackage{a4wide,color,eucal,enumerate,mathrsfs}
\usepackage[normalem]{ulem}
\usepackage{amsmath,amssymb,epsfig,amsthm} 
\usepackage[latin1]{inputenc}
\usepackage{psfrag}
\numberwithin{equation}{section}

\newtheorem{theorem}{Theorem}[section]

\newtheorem{corollary}[theorem]{Corollary}
\newtheorem{lemma}[theorem]{Lemma}

\theoremstyle{remark}

\DeclareMathOperator{\dist}{dist}

\newcommand{\N}{\mathbb{N}}

\def\dist{{\mathop\mathrm{\,dist\,}}}

\def\bint{{\ifinner\rlap{\bf\kern.35em--}
\int\else\rlap{\bf\kern.45em--}\int\fi}\ignorespaces}

\def\bbint{{\ifinner\rlap{\bf\kern.35em--}
\hspace{0.078cm}\int\else\rlap{\bf\kern.45em--}\int\fi}\ignorespaces}

\def\bint{{\ifinner\rlap{\bf\kern.35em--}
\int\else\rlap{\bf\kern.45em--}\int\fi}\ignorespaces}

\begin{document}

\title[Approximation by uniform domains in doubling quasiconvex metric spaces]
{Approximation by uniform domains in\\ doubling quasiconvex metric spaces}

\author{Tapio Rajala}

\address{University of Jyvaskyla \\
         Department of Mathematics and Statistics \\
         P.O. Box 35 (MaD) \\
         FI-40014 University of Jyvaskyla \\
         Finland}
\email{tapio.m.rajala@jyu.fi}

\thanks{The author acknowledges the support from the Academy of Finland, grant no. 314789. 
}
\subjclass[2000]{Primary 30L99. Secondary 46E35, 26B30.}
\keywords{Sobolev extension, uniform domain, quasiconvexity}
\date{\today}



\maketitle

\begin{center}
\emph{Dedicated to Professor Pekka Koskela on the occasion of his 60th birthday.}
\end{center}

\begin{abstract}
We show that any bounded domain in a doubling quasiconvex metric space can be approximated from inside and outside by uniform domains.
\end{abstract}

\section{Introduction}

We provide an approximation of bounded domains from inside and from outside by uniform domains in doubling quasiconvex metric spaces. A metric space $(X,d)$ is called \emph{(metrically) doubling}, if there exists a constant $C_d$ so that for all $r>0$, any ball of radius $r$ can be covered by $C_d$ balls of radius $r/2$. A metric space is called \emph{quasiconvex}, if there exists a constant $C_q < \infty$ such that any $x,y \in X$ can be connected by a curve $\gamma$ in $X$ with the length bound
\[
 \ell(\gamma) \le C_q d(x,y).
\]
A domain $\Omega \subset X$ is called \emph{uniform}, if there exists a constant $C_u < \infty$ such that for every $x,y \in \Omega$ there exists a curve $\gamma \subset \Omega$ such that
\[
 \ell(\gamma) \le C_u d(x,y)
\]
and for all $z \in \gamma$ it holds
\[
 \min\left\{\ell(\gamma_{x,z}),\ell(\gamma_{z,y})\right\} \le C_u\dist(z,X\setminus \Omega),
\]
where $\gamma_{x,z}$ and $\gamma_{z,y}$ denote the shortest subcurves of $\gamma$ joining $z$ to $x$ and $y$, respectively. 

With the definitions now recalled we can state the result of this paper.

\begin{theorem}\label{thm:main}
 Let $(X,d)$ be a doubling quasiconvex metric space and $\Omega \subset X$ a bounded domain. Then for every $\varepsilon > 0$ there exist uniform domains
 $\Omega_I$ and $\Omega_O$ such that
 \[
  \Omega_I \subset \Omega \subset \Omega_O,
 \]
 $\Omega_O \subset B(\Omega,\varepsilon)$,
 and $X \setminus \Omega_I \subset B(X \setminus \Omega,\varepsilon)$.
\end{theorem}
In the above theorem we have used the notation
\[
B(A,r) = \bigcup_{x \in A} B(x,r)
\]
for the open $r$-neighbourhood of a set $A \subset X$, with $r > 0$, and $B(x,r)$ denoting the open ball of radius $r$ centred at a point $x \in X$.

Although there are characterizations of uniform domains in metric spaces, for instance via tangents \cite{H2011}, we are not aware of previous general existence results such as Theorem \ref{thm:main}.
\medskip

The setting of Theorem \ref{thm:main} is motivated by Sobolev- and BV-extension domains in complete metric measure spaces with a doubling measure and supporting a $(1,p)$-Poincar\'e inequality ($p$-PI-spaces for short). A measure $\mu$ on $(X,d)$ is doubling, if there exists a constant $C>0$ such that $\mu(B(x,2r)) \le C\mu(B(x,r))$ for every $x \in X$ and $r>0$. Recall that a metric space supporting a positive and locally finite doubling measure $\mu$ is doubling in the metric sense. The metric measure space $(X,d,\mu)$ supports a $(1,p)$-Poincar\'e inequality if there exist constants $C,\lambda\ge 1$ so that the following holds: for any $x \in X$ and $r>0$ the ball $B(x,r) \subset X$ has positive and finite $\mu$-measure and 
\[
 \frac{1}{\mu(B(x,r))}\int_{B(x,r)}|u-u_{B(x,r)}|\,d\mu \le C r \left(\frac{1}{\mu(B(x,\lambda r))}\int_{B(x,\lambda r)} \rho^p \right)^\frac{1}{p}
\]
holds for any measurable function $u$ and its upper gradient $\rho$, with $u_{B(x,r)}$ being the average of $u$ in $B(x,r)$.

On one hand, $p$-PI spaces \cite{HK1998} are known to be quasiconvex \cite{C1999,K2003}. On the other hand, in \cite{BS2007} it was shown that uniform domains in $p$-PI-spaces are $N^{1,p}$-extension domains, for $1 \le p < \infty$, for the Newtonian Sobolev spaces, and in \cite{L2015} it was shown that bounded uniform domains in 1-PI-spaces are BV-extension domains. See \cite{S2000} for the definitions of upper gradients and Newtonian Sobolev spaces and \cite{A2002,M2003} for the BV space.
The main purpose of this paper is to increase the applicability of the results in \cite{BS2007,L2015} by providing a large collection of uniform domains.  As a straightforward corollary we have the following approximation result by extension domains.

\begin{corollary}\label{cor:extension}
 Let  $1 \le p < \infty$, let $(X,d,\mu)$ be a complete metric measure space, with $\mu$ doubling, supporting a (1,p)-Poincar\'e inequality, and let $\Omega \subset X$ be a bounded domain. Then $\Omega$ can be approximated (as in Theorem \ref{thm:main}) by $N^{1,p}$-extension and, in the case $p=1$, also by $BV$-extension domains.
\end{corollary}

Notice also that in the case when $\Omega$ is unbounded, we can for example fix a point $x_0 \in \Omega$ and for each $i \in \N$ approximate the connected component of $B(x_0,i) \cap \Omega$ containing $x_0$ from inside by $\Omega_i$ using Theorem \ref{thm:main} with the choice $\varepsilon = 1/i$, and thus obtain
\[
 \Omega = \bigcup_{i=1}^\infty \Omega_i,
\]
with $\Omega_i$ uniform for all $i \in \N$.

\section{Construction of the uniform domains}

In the Euclidean setting we could use closed dyadic cubes to construct the uniform domains. Using just the fact that a Euclidean cube is John (and not that it is in fact uniform), we could start with a finite union of cubes of some fixed side-length, then take all the neighbouring cubes with a constant $c \in (0,1)$ times smaller side-length than the original ones and continue taking smaller and smaller cubes.
The main thing one has to take care about is that two points near the boundary that are some small distance $r$ from each other can be connected by going via cubes not much larger than $r$ in side-length. This is handled by taking the constant $c$ small enough because of the nice property of closed Euclidean dyadic cubes: if two cubes of side-length $l$ do not intersect, then their distance is at least $l$.

We will use the above idea in the metric setting. However, none of the dyadic cube constructions that we have seen (for instance \cite{C1990,HK2012,HM2012,KRS2012}) take care about the separation of non-intersecting cubes, but only about other properties such as nestedness and size. Luckily, we do not need a nested structure, nor a decomposition, so we will work with coverings by balls having the needed separation property. The existence of such coverings is provided by the next lemma.

\begin{lemma}\label{lma:separatedballs}
 Let $(X,d)$ be a doubling metric space. Then there exists a constant $c \in (0,1)$ depending only on the doubling constant so that for every $r>0$ there exist $r$-separated points $\{x_i\} \subset X$ and radii $r_i \in [r,2r]$ such that
 \[
  X \subset \bigcup_{i}B(x_i,r_i)
 \] 
 and
 \[
  d(x_i,x_j) -r_i -r_j \notin (0,cr) \qquad \text{for all }i,j.
 \]
\end{lemma}
\begin{proof}
 Let $\{x_i\}$ be a maximal $r$-separated net of points in $X$.
 Because of the maximality of the net, the balls $B(x_i,r_i)$ will cover $X$. 
 We select the suitable radii by induction. Let $r_1 = r$. Suppose that $r_1, \dots, r_k$ have been selected. Since $x_i$ are $r$-separated, by the metric doubling property of $(X,d)$, there exists an integer $N > 1$ depending only on the doubling constant $C_d$ so that there exist at most $N-1$ points 
 $x_i \in \{x_1, \dots, x_k\}$ with $d(x_{k+1},x_i) \le 4r$.
 Write 
 \[
  I_k = \left\{i \,:\, d(x_i,x_{k+1}) - r_i \in [r,2r], i \le k\right\}.
 \]
 Then $I_k$ contains at most $N-1$ points. Let $\lambda_1 < \lambda_2 < \cdots <\lambda_M$, with $M < N$, be so that 
 \[
 \{\lambda_j\}_{j=1}^M = \left\{d(x_i,x_{k+1}) - r_i \,:\, i \in I_k\right\}.
 \]
 Denote $\lambda_0 = r$ and $\lambda_{M+1} = 2r$. Let $m \in \{0, \dots, M\}$ be the smallest integer for which
 $\lambda_{m+1} - \lambda_{m}\ge r/N$. (If such $m$ did not exist we would have
\[
 r = \lambda_{M+1} - \lambda_0 = \sum_{j=0}^M \lambda_{j+1}-\lambda_j < (M+1)\frac{r}{N} \le r,
\]
which is a contradiction.) We now define $r_{k+1} =  \lambda_m$. In particular, we then have 
\begin{equation}\label{eq:boundfrom2r}
 r_{k+1} \in [r,2r-r/N].
\end{equation}
By the definition of $m$, we have
 \[
  d(x_i,x_{k+1}) - r_i - r_{k+1} \notin (0,r/N)
 \]
 for all $i \in I_k$. 
 
 Now, if $i \le k$ with $i \notin I_k$, either $d(x_i,x_{k+1}) - r_i < r$, in which case $  d(x_i,x_{k+1}) - r_i - r_{k+1} < 0$, or
 $d(x_i,x_{k+1}) - r_i > 2r$, in which case by \eqref{eq:boundfrom2r} we have $d(x_i,x_{k+1}) - r_i - r_{k+1} > r/N$.
  Thus, 
 \[
  d(x_i,x_{k+1}) - r_i - r_{k+1} \notin (0,r/N)
 \]
 for all $i \in\{1,\dots,k\}$. This shows that the claim holds with the constant $c = 1/N$.
\end{proof}

With the replacement of the Euclidean dyadic cubes by balls given in Lemma \ref{lma:separatedballs} we can now follow the idea presented for the Euclidean case to prove the metric version.

\begin{proof}[Proof of Theorem \ref{thm:main}]
 We start by noting that since our space $(X,d)$ is quasiconvex,
 the induced length distance
 \[
  d_l(x,y) = \inf\left\{\ell(\gamma)\,:\,\text{the curve } \gamma \text{ joins }x \text{ to }y\right\}
 \]
  satisfies $d \le d_l \le C_qd$ with the quasiconvexity constant $C_q$.
  If we would assume the space $(X,d)$ to be complete, by the generalized Hopf-Rinow Theorem we would know that $d_l$ is in fact a geodesic distance. However, we want to avoid making the extra assumption on completeness.
  In any case, because the property of being a uniform domain is invariant under a biLipschitz change of the distance, we may then assume that $(X,d)$ is a length space.
  \medskip
  
  {\color{blue}\textsc{Construction:}}
  The constructions of $\Omega_I$ and $\Omega_O$ are similar. The only difference is the starting point of the construction.
  Fix a point $x_0 \in \Omega$ and let $\tau \in (0,\min\{\dist(x_0,\partial\Omega),1\})$. The choice of $\tau$ will depend on $\varepsilon$, and the estimate on how small $\tau$ we need to select is postponed to the end of the proof. For constructing $\Omega_O$ we simply start with the set
  \[
   E_1 = \Omega,
  \]
  and for $\Omega_I$ we take $E_1$ to be the connected component of 
  \begin{equation}\label{eq:E_1}
    \left\{x \in X\,:\,\dist(x,X\setminus \Omega) > \tau\right\}
  \end{equation}
  containing the fixed point $x_0$.
  Let us consider the case $\Omega_I$. Thus $E_1$ is defined via \eqref{eq:E_1}.
  
  Let $c \in (0,1)$ be the constant from Lemma \ref{lma:separatedballs}. Define
  \[
   \delta=\min\left\{\frac{c}{20+c},\frac{\tau}{5+\tau}\right\}.
  \]
  We construct $\Omega_I$ using induction as follows. Suppose $E_k$ has been defined for a $k \in \N$. Let $\{x_i\}$ and $\{r_i\}$ be the points and radii given by Lemma \ref{lma:separatedballs} for the choice $r = \delta^k$, and define
  \[
   \mathcal B_k = \left\{B(x_i,r_i) \,:\, B(x_i,r_i)\cap B(E_k,\delta^k) \ne \emptyset \right\}.
  \]
  We then set 
  \[
   E_{k+1} = \bigcup_{B(x,r) \in \mathcal B_k} B(x,r).
  \]
  Finally, we define
  \[
   \Omega_I = \bigcup_{k=1}^\infty E_k.
  \]
  \medskip

  {\color{blue}\textsc{Uniformity:}}
  Let us next show that $\Omega_I$ is uniform. Take $x,y \in \Omega_I$ with $x \ne y$.
  Let $k_x$ and $k_y$ be the smallest integers such that $x \in E_{k_x}$ and $y \in E_{k_y}$. Without loss of generality we may assume $k_x \le k_y$.
  
  Suppose first that $d(x,y) < \frac14c\delta$. Let $n \in \N$ be such that
  \[
   \frac14c\delta^{n+1} \le d(x,y) <\frac14c\delta^n.
  \]
  Notice that since in each construction step $k+1$ we take a neighbourhood $\delta^k$ of the previous set $E_k$, we have that
  \begin{equation}\label{eq:bdrydist}
    \dist(E_k, X \setminus \Omega_I) \ge \sum_{i=k}^\infty\delta^i = \frac{\delta^k}{1-\delta} > \delta^k.
  \end{equation}
  Therefore, if $k_x < n$, we may take $\gamma$ to be a curve connecting $x$ to $y$ so that $\ell(\gamma) < 2 d(x,y)$, in which case for all $z \in \gamma$ we have
  \[
 \dist(z,X \setminus \Omega_I) \ge  \dist(x,X \setminus \Omega_I)- d(z,x) >\delta^{k_x} - \ell(\gamma) > 4d(x,y) - 2d(x,y) 
   \ge 2d(x,y),
  \]
  and, consequently, we get uniformity with constant $C_u = 2$.
  
  If $k_x \ge n$, we first connect $x$ and $y$ to $E_n$. We do this as follows. Starting with $x$, 
  let $B(z,r) \in \mathcal B_{k_x-1}$ be such that $x \in B(z,r)$,
  which exists by the definitions of $k_x$ and $E_{k_x}$.
  Next take $v \in B(z,r) \cap B(E_{k_x-1}, \delta^{k_x-1})$ and  $w \in E_{k_x-1}$ with $d(v,w) < \delta^{k_x-1}$, which we have by the definition of $\mathcal B_{k_x-1}$. Now we take the concatenation $\gamma_{k_x}^x$ of curves $\alpha_1$ going from $x$ to $z$, $\alpha_2$ going from $z$ to $v$ and $\alpha_3$ going from $v$ to $w$ with the length bounds
  $\ell(\alpha_1), \ell(\alpha_2) < r$ and $\ell(\alpha_3)<\delta^{k_x-1}$. Notice that $\gamma_{k_x}^x\subset E_{k_x}$ and that the curve $\gamma_{k_x}^x$ has the length bound
  \[
   \ell(\gamma_{k_x}^x) < r + r + \delta^{k_x-1} \le 5\delta^{k_x-1}.
  \]
  For the distance to the complement of $\Omega_I$ we can estimate
  \begin{equation}\label{eq:distboundary}
   \dist(\gamma_{k_x}^x,X\setminus\Omega_I) >  \delta^{k_x}
  \end{equation}
  by the fact that in the construction of $E_{k_x+1}$ we take a $\delta^{k_x}$-neighbourhood of $E_{k_x}$ and the curve $\gamma_{k_x}^x$ is contained in $E_{k_x}$. We then continue inductively connecting $w$ to $E_{k_x-2}$ by $\gamma_{k_x-1}^x$ and so on, until we have connected $x$ to a point $x'$ in $E_n$.
  
  The curve $\gamma^{x,x'}$ obtained by concatenating the previous curves $\gamma_{k_x}^x, \gamma_{k_x-1}^x, \dots, \gamma_{n+1}^x$ has the length bound
  \begin{equation}\label{eq:gxxbound}
   \ell(\gamma^{x,x'}) \le \sum_{i=n}^{k_x-1}5\delta^i \le 5\frac{\delta^n}{1-\delta} \le \frac14c\delta^{n-1}.
  \end{equation}
  With a similar construction, we connect $y$ to a point $y' \in E_n$ by a curve $\gamma^{y,y'}$ with length bounded from above by $c\delta^{n-1}/4$.
  We can bound the distance between $x'$ and $y'$ by
  \begin{equation}\label{eq:Endist}
   d(x',y') \le d(x',x)+d(x,y)+d(y,y')
   < \frac14c\delta^{n-1} + \frac14c\delta^n + \frac14c\delta^{n-1} < c\delta^{n-1}.
  \end{equation}
 
 Now we  use the crucial separation property given by Lemma \ref{lma:separatedballs}.
Let $B(z_x,r_x),B(z_y,r_y) \in \mathcal{B}_{n-1}$ be such that $x' \in B(z_x,r_x)$ and $y' \in B(z_y,r_y)$.
 Since the collection $\mathcal{B}_{n-1}$ was defined via Lemma \ref{lma:separatedballs} with the radius $\delta^{n-1}$, we have 
 \[
d(z_x,z_y) - r_x-r_y \notin (0,c\delta^{n-1}),  
 \]
 whereas \eqref{eq:Endist} gives 
 \[
  d(z_x,z_y) - r_x-r_y \le d(z_x,x') + d(x',y') + d(y',z_y)- r_x-r_y \le d(x',y') < c\delta^{n-1}.
 \]
Therefore, $d(z_x,z_y) \le r_x+r_y$ and thus we can connect $x'$ to $y'$ by a curve $\gamma^{x',y'}$ defined by going first with a curve $\beta_1$ from $x'$ to $z_x$, then with $\beta_2$ from $z_x$ to $z_y$ and finally with $\beta_3$ from $z_y$ to $y'$.  By selecting the curves so that $\ell(\beta_1) < r_x$,  
\[
\ell(\beta_2) < r_x + r_y + \min\left\{r_x - \ell(\beta_1) , \delta^{n+1}\right\},
\]
 and $\ell(\beta_3) < r_y$, the curve $\gamma^{x',y'}$ has the length bound
 \begin{equation}\label{eq:gxybound}
  \ell(\gamma^{x',y'}) \le \ell(\beta_1) + \ell(\beta_2) + \ell(\beta_3) <  2r_x + 2r_y \le 8\delta^{n-1},
 \end{equation}
 and its distance to the complement of $\Omega_I$ has the bound
 \begin{equation}\label{eq:xydist}
  \dist(\gamma^{x',y'},X \setminus\Omega_I) > \delta^n.
 \end{equation}

  Now, the curve $\gamma$ obtained by concatenating $\gamma^{x,x'},\gamma^{x'y'}$ and $\gamma^{y,y'}$ has, by \eqref{eq:gxxbound} and \eqref{eq:gxybound}, length at most
  \begin{equation}\label{eq:gammalength}
  \ell(\gamma) \le  \frac14 c\delta^{n-1} + 8\delta^{n-1} +  \frac14 c\delta^{n-1} \le 
   9\delta^{n-1} = \frac{36}{c\delta^2} \cdot \frac14c\delta^{n+1} \le \frac{36}{c\delta^2} d(x,y).
  \end{equation}
  
  Let us check the uniformity for this curve. Let $z \in \gamma$. Suppose first that $z \in \gamma^{x',y'}$. Then by \eqref{eq:xydist} and \eqref{eq:gammalength}, we get
  \begin{equation}\label{eq:est1}
  \min\left\{\ell(\gamma_{x,z}),\ell(\gamma_{z,y})\right\} \le \frac12\ell(\gamma) \le \frac92\delta^{n-1} = \frac{9}{2\delta} \delta^n \le \frac{9}{2\delta}\dist(z,X\setminus \Omega_I).
 \end{equation}
  By symmetry it then remains to check the case $z \in \gamma^{x,x'}$. Then there exists $k \ge n$ such that $z \in \gamma_k^x$. Then by \eqref{eq:distboundary} and the same estimate as in \eqref{eq:gxxbound}, we get
  \begin{equation}\label{eq:est2}
  \min\left\{\ell(\gamma_{x,z}),\ell(\gamma_{z,y})\right\} \le
  5\frac{\delta^k}{1-\delta} \le 10\delta^k  \le
  10 \dist(z,X\setminus \Omega_I).
 \end{equation}
  By combining the estimates \eqref{eq:gammalength}, \eqref{eq:est1} and \eqref{eq:est2} we see that $\gamma$ satisfies the uniformity condition with the constant $C_u = 36/(c\delta^2)$.

  We are still left with proving the uniformity in the case $d(x,y) \ge \frac14c\delta$. For this we first observe that we can connect $x$ to a point $x' \in E_1$, and $y$ to a point $y' \in E_1$ by curves having lengths bounded from above by $c/4$ and with pointwise lower-bounds for the distance to the boundary along the curves being enough for the uniformity condition. What remains to do is to connect $x'$ to $y'$ with a curve whose length is bounded by a constant (independent of $x'$ and $y'$) from above and whose distance to the boundary of $\Omega_I$ is bounded by another constant from below. This is achieved directly by compactness: on one hand, any two points in the totally bounded set ${E_1}$ can be joined by a rectifiable curve inside $B(E_1,\delta/2) \subset \Omega_I$ and the infimum over the lengths of curves joining two given points is a continuous function in terms of the endpoints and this function extends to the completion of $E_1$ as a continuous function. Thus, there exists the needed constant upper bound for the lengths of curves. On the other hand, the distance of these curves to the boundary of $\Omega_I$ is at least $\delta/2$.
  \medskip
  
  {\color{blue}\textsc{Closeness:}}
  Let us then show that for every $\varepsilon>0$ there exists $\tau>0$ so that using the $\tau$ in the construction above we get $X \setminus \Omega_I \subset B(X \setminus \Omega,\varepsilon)$.
  
  In order to have the dependence on $\tau$, write now $E_1(\tau)$ to be the connected component of
  $\left\{x \in X\,:\,\dist(x,X\setminus \Omega) > \tau\right\}$
  containing $x_0$. 
  
  Since $X \setminus B(X \setminus \Omega,\varepsilon)$ is totally bounded, there exists a set of points $\{x_i\}_{i=1}^N \subset \Omega$ so that 
  \[
   X \setminus B(X \setminus \Omega,\varepsilon) \subset \bigcup_{i=1}^N B(x_i,\varepsilon / 2).
  \]
  Each $x_i$ can be connected to $x_0$ by a curve inside $\Omega$ and so there exist $\tau_i>0$ for which $x_i \in E_1(\tau_i)$. Consequently, with $\tau = \min\{\varepsilon/2,\tau_1, \dots, \tau_N\}$ we have
  $X \setminus B(X \setminus \Omega,\varepsilon) \subset E_1(\tau)$, and thus, 
  \[
   X \setminus \Omega_I \subset X \setminus E_1(\tau) \subset B(X \setminus \Omega,\varepsilon).
  \]
  
  The final thing we still need to observe is that $\Omega_I \subset \Omega$.
  By the construction procedure, we have
  \[
   E_{k+1} \subset B(E_k,5\delta^k)
  \]
  for every $k \in \N$. Thus, by the choice of $\delta$ we get
  \[
   \Omega_I \subset B(E_1,\sum_{k=1}^\infty 5\delta^k) \subset B(E_1,\tau) \subset \Omega.
  \]
  This completes the proof for $\Omega_I$. The proof for $\Omega_O$ goes almost verbatim. Only the argument for closeness becomes easier in this case. In particular, for $\Omega_O$ one can then take $\tau = \varepsilon$.
\end{proof}

\section*{Acknowledgments}
The author thanks Nageswari Shanmugalingam for bringing this question to his attention, and for the related discussions during the IMPAN conference \emph{Latest in Geometric Analysis. Celebration of Pekka Koskela's 59th birthday} in November 2019. The author also thanks Anders Bj\"orn, Jana Bj\"orn, Panu Lahti, and the anonymous referee for the corrections and comments on the previous version of the paper.


\end{document}